\documentclass[lettersize,journal]{IEEEtran}
\usepackage{amsmath,amsfonts,amssymb}
\usepackage{algpseudocode}
\usepackage{algorithm}
\usepackage{array}
\usepackage{subfigure}
\usepackage{textcomp}
\usepackage{stfloats}
\usepackage{url}
\usepackage{verbatim}
\usepackage{graphicx}
\usepackage{balance}

\usepackage{cite}
\usepackage{bm}
\usepackage{color}  
\usepackage{orcidlink}
\def\BibTeX{{\rm B\kern-.05em{\sc i\kern-.025em b}\kern-.08em
    T\kern-.1667em\lower.7ex\hbox{E}\kern-.125emX}}
\usepackage{amsthm}

\begin{document}
% 声明新的块
\algblock{ParFor}{EndParFor}

% 自定义新的块
\algnewcommand\algorithmicparfor{\textbf{parfor\:}}
\algnewcommand\algorithmicpardo{\textbf{do}}
\algnewcommand\algorithmicendparfor{\textbf{end parfor}}

\algrenewtext{ParFor}[1]{\algorithmicparfor #1 \algorithmicpardo}
\algrenewtext{EndParFor}{\algorithmicendparfor}

\title{Distributed Model Predictive Control for Energy and Comfort Optimization in Large Buildings Using Piecewise Affine Approximation}

\author{Hongyi Li \orcidlink{0000-0003-3585-7365}, Jun Xu \orcidlink{0000-0002-2934-4814},~\IEEEmembership{Senior Member,~IEEE}, Jinfeng Liu  \orcidlink{0000-0001-8873-847X}
\thanks{
This work was supported  in part by the National Key Research and Development Program of China under Grant 2025YFE0101100, in part by the Science and Technology Innovation Committee of Shenzhen Municipality under Grant GXWD20231129101652001, in part by the Shenzhen Science and Technology Program under Grant SYSPG20241211173609005, and in part by Guangdong Provincial Key Laboratory of Intelligent Morphing Mechanisms and Adaptive Robotics under Grant 2023B1212010005.
(Corresponding author: Jun Xu.)
}
\thanks{
Hongyi Li and Jun Xu are with the School of Intelligence Science and Engineering, Harbin Institute of Technology, Shenzhen, 518055, China 
and with the Shenzhen Key Lab for Advanced Motion Control and Modern Automation Equipments, 
Shenzhen, 518055, China
(email: 23b904015@stu.hit.edu.cn; xujunqgy@hit.edu.cn).

Jinfeng Liu is with the Department of Chemical \& Materials Engineering, 
University of Alberta, Edmonton, Alberta, Canada, T6G 1H9 
(email: jinfeng@ualberta.ca).
}
}

\maketitle
\newtheorem{lemma}{Lemma}
\newtheorem{remark}{Remark}
\newtheorem{theorem}{Theorem}
\newtheorem{assumption}{Assumption}
\begin{abstract}
\textcolor{black}{
The control of large buildings encounters challenges in computational efficiency due to their size and nonlinear components. To address these issues, this paper proposes a Piecewise Affine (PWA)-based distributed scheme for Model Predictive Control (MPC) that optimizes energy and comfort through PWA-based quadratic programming. We utilize the Alternating Direction Method of Multipliers (ADMM) for effective decomposition and apply the PWA technique to handle the nonlinear components. To solve the resulting large-scale nonconvex problems, the paper introduces a convex ADMM algorithm that transforms the nonconvex problem into a series of smaller convex problems, significantly enhancing computational efficiency. Furthermore, we demonstrate that the convex ADMM algorithm converges to a local optimum of the original problem. A case study involving 36 zones validates the effectiveness of the proposed method. Our proposed method reduces execution time by 86\% compared to the centralized version.}
\end{abstract}


\section{Introduction}
Building energy consumption accounts for approximately 40\% of global energy use \cite{9478948}, with a significant proportion of that attributed to Heating, Ventilation, and Air Conditioning (HVAC) systems \cite{8956087}. In response to the growing energy crisis, many researchers are exploring solutions to reduce energy consumption (see \cite{9478948,8956087,7087366}), and improving the efficiency of HVAC systems holds significant potential. Energy efficiency in HVAC systems can typically be enhanced via improved building materials and design, as well as through the optimization of HVAC control strategies. However, it is expensive to retrofit building structures and materials for existing buildings, while employing advanced control methods to optimize HVAC operation offers a more cost-effective solution \cite{7087366}.

Model Predictive Control (MPC) is a promising alternative to traditional control methods in the building sector, valued for its ability to manage constraints, predict dynamics, and optimize multiple objectives \cite{8745685}. However, centralized MPC struggles with scalability and computational efficiency in large-scale systems, where growing decision variables hinder real-time operation. 
Distributed MPC enhances scalability and flexibility by decomposing complex control problems into smaller, independent subproblems, each managed by a local controller \cite{10011547}.
In this framework, the local controllers communicate to ensure overall system performance \cite{christofides2013distributed}. 
Distributed MPC has been applied across various fields, including energy systems \cite{8643528,li2024priority}, and multi-robot systems \cite{10381783,10004912,9508858}. In these domains, coordinating subsystems is essential for optimizing overall performance. By allowing local controllers to make independent decisions while maintaining system-wide consistency, distributed MPC has demonstrated its effectiveness in handling complex, interconnected environments.
A key challenge in distributed MPC is addressing the couplings between subsystems, which often require careful coordination among multiple controllers. The Alternating Direction Method of Multipliers (ADMM) is particularly effective at decomposing optimization problems with such coupling constraints, making it well-suited for distributed MPC. Several studies have investigated the integration of ADMM with distributed MPC in building control.
For instance, distributed MPC via proximal Jacobian ADMM, a variant of the traditional ADMM, is applied in \cite{7962927} for HVAC energy optimization.
A distributed MPC scheme, considering the dynamics of producers and consumers simultaneously, is investigated in \cite{9837140} based on ADMM. \cite{9733935} introduce an ADMM-based distributed stochastic MPC framework for a network of building-integrated microgrids.

In addition, ensuring a comfortable indoor environment is a critical aspect of building control. While most existing work focuses on indoor temperature alone (see \cite{8668526,7000560,9415466}), the Predicted Mean Vote (PMV) index provides a more comprehensive evaluation by incorporating indoor temperature, humidity, clothing, wind speed, and other factors. Despite its comprehensiveness, \textcolor{black}{the nonlinear component in the PMV formula} introduces a significant computational burden, making its direct application in MPC challenging. To address this, some studies work on approximating the PMV index. A linear approximation is employed in \cite{cigler2012optimization,yang2018state,yang2020model} of PMV to simplify calculations. However, this approach introduces larger deviations from true values at points far from the linearization point, potentially affecting controller performance. To further improve accuracy, recent work has explored more sophisticated models like Piecewise Affine (PWA) representations of the PMV index~\cite{li2023economic}. While such high-fidelity models can improve controller performance, their implementation often relies on centralized architectures. This centralized nature imposes considerable computational demands on large-scale building systems, thus limiting practical scalability and posing a significant challenge in the field.
To address this scalability challenge, this paper proposes a novel distributed control framework designed to efficiently apply high-fidelity PWA comfort models in large-scale buildings.

%\textcolor{black}{One of the contributions of this paper is to extend the PWA approximation technique to comfort modeling in large buildings, enhancing its applicability and performance in more complex environments.
%Furthermore, we propose a new distributed algorithm based on ADMM decomposition and PWA techniques to efficiently solve the resulting large-scale nonconvex optimization problem.}

The main contributions of this paper are as follows:
\begin{itemize}
    \item 
    We propose a PWA-based distributed MPC scheme for large building control. Unlike conventional methods, this scheme utilizes \textcolor{black}{PWA-based quadratic programming} to address comfort and energy optimization. To tackle the resulting large-scale nonconvex problem, we introduce a convex ADMM algorithm to solve it efficiently.

    \item The proposed convex ADMM algorithm provides a novel approach to solving PWA-based quadratic programming problems in large-scale nonconvex scenarios. By leveraging the ADMM decomposition and PWA structure, the algorithm achieves performance close to that of centralized methods, while only solving a sequence of small-scale convex problems. This method significantly improves computational efficiency. Additionally, we demonstrate that the convex ADMM algorithm converges to a local optimum of the original problem.

    \item  To our knowledge, the application of PWA approximation techniques to the PMV index within a distributed MPC framework has not been widely reported. This paper aims to address this gap. In particular, we apply a PWA approximation of the comfort index PMV to each zone in the PWA-based distributed scheme. These individual approximations are then integrated and serve as one of the global objectives in this distributed scheme. \textcolor{black}{Additionally, our proposed distributed technique is implemented in a scenario involving 36 zones, demonstrating its scalability and practical benefits.}

\end{itemize}

The rest of this paper is organized as follows:
Section \ref{s2} provides an overview of the framework and model details. Section \ref{s3} presents the PWA-based distributed MPC scheme and the novel convex ADMM algorithm. Section \ref{s5} includes a case study involving 36 zones to illustrate the advantages of the proposed methods. 
Section \ref{sec:discussion} discusses the advantages and limitations of the proposed approach compared with existing methods. Finally, Section \ref{s6} concludes the paper.

\section{System descriptions}

\label{s2}
\textcolor{black}{This section provides system descriptions that serve as a foundation for the subsequent discussion. We begin by outlining the general framework of distributed optimization presented in this paper. Next, we describe the prediction model utilized for distributed optimization. Finally, we introduce the \textcolor{black}{modeling of the comfort index in the MPC cost function} employed in this study.}

\begin{figure}[h]
\vskip -0.1in
	\centering
	\includegraphics[width=0.5\textwidth]{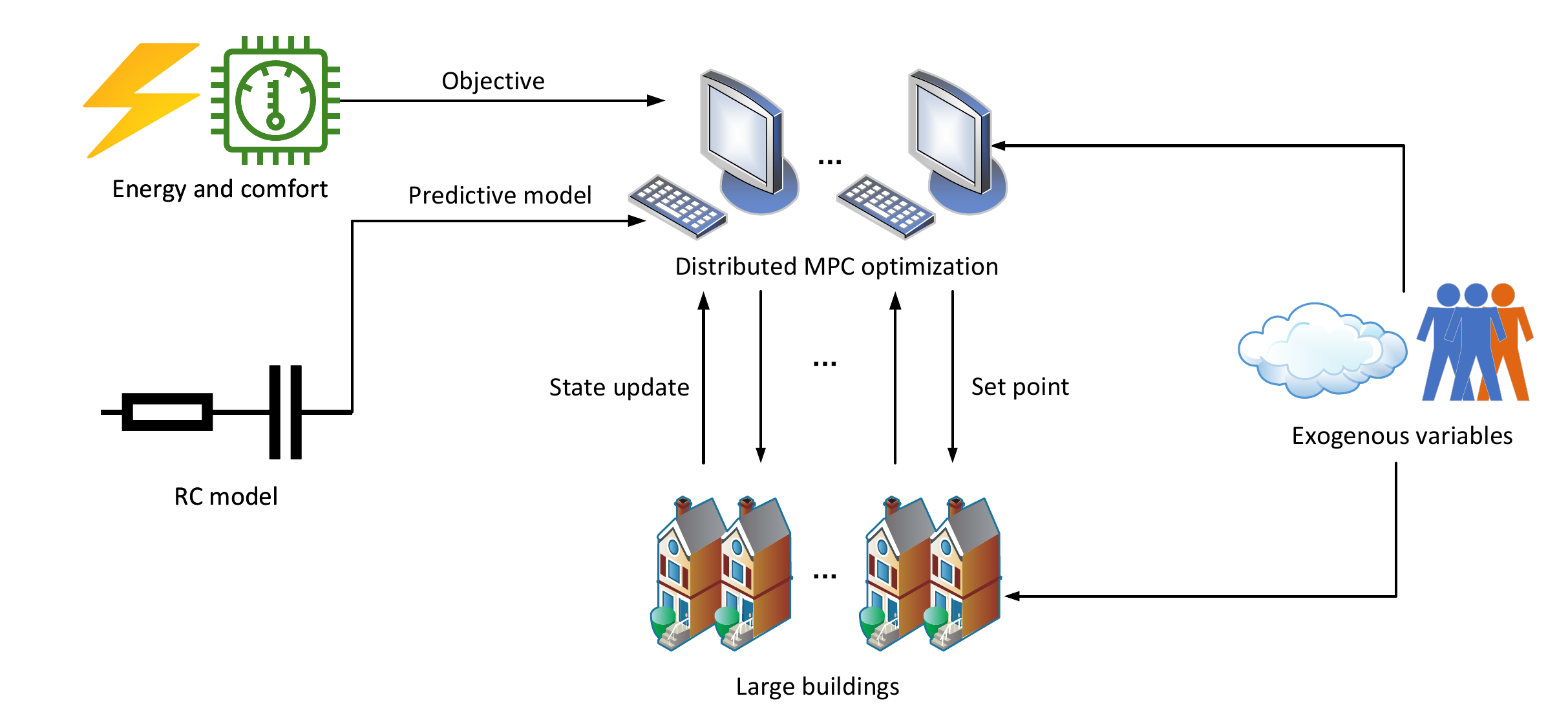} 
	\caption{\textcolor{black}{Overall framework for distributed optimization of energy and comfort in large buildings.}} 
	\label{general_framework_distributedADMM} 
    \vskip -0.2in
\end{figure}

\subsection{Overall framework}
The overall framework of this paper is illustrated in Fig. \ref{general_framework_distributedADMM}. \textcolor{black}{The control objective is to minimize energy consumption while simultaneously enhancing occupant comfort in large buildings. This is achieved by a distributed optimization scheme, in which a series of distributed MPC optimizations are designed to balance energy efficiency and occupant comfort.}
\textcolor{black}{These optimizations are performed in MATLAB to determine the set points for each zone, which are then applied to large buildings in EnergyPlus \cite{crawley2001energyplus}.} The state of the controlled zones in EnergyPlus is updated to the distributed MPC scheme in real time before each optimization step is executed. 
Additionally, the study considers external factors, such as occupancy patterns, weather conditions, and other relevant disturbances, to enhance the practical applicability of the proposed method.

\subsection{Predictive model for MPC}
The RC model, a mechanistic approach, is widely employed to simulate and predict the thermal behavior of buildings. In this model, thermal resistance simulates the transfer of heat through building materials, while heat capacity represents the building’s ability to store heat. \textcolor{black}{The RC model is formulated by constructing differential equations for indoor and wall nodes. }

In this paper, RC models are used as the predictive tool for MPC.
Table \ref{tab2} presents the notations for the quantities about RC models discussed in this paper.
\begin{table}[htbp] 
\vskip -0.1in
    \centering 
    \caption{Quantities and their notations} 
    \label{tab2} 
    \scriptsize
    \begin{tabular}{cc} 
    \hline Notation & Description \\
    \hline $u$ & HVAC system input \\ 
    $x$ & HVAC system state \\ 
    $ori$ & Represents directions: $n$ (north), $e$ (east), $w$ (west), $s$ (south) \\ 
    $C_{z}$ & Zone thermal capacity \\
    $C_{ {ori }}^{w}$ & Thermal capacity of the wall \\ 
    $R_{ori}^w$ & Thermal resistance for wall conduction \\ 
    $R_{ori}$ & Thermal resistance for convection on the inner surface \\ 
    $R_{ori}^\prime$ & Thermal resistance for convection on the outer surface \\ 
    $T_{z}$ & Indoor air temperature \\ 
    $T_{ori}^{wi}$ & Inner surface temperature of the $ori$ wall \\ $T_{ori}^{wo}$ & Outer surface temperature of the $ori$ wall \\ $T_{ori}^{out}$ & Outside temperature adjacent to the $ori$ wall  \\ 
    $\dot{Q}_{ori}^{rad}$ & Solar radiation incident on the $ori$ wall \\ 
    $\dot{Q}_{z}^{in}$ & Internal gains within the zone \\ 
    $\dot{Q}_{z}^{rad}$ & Solar radiation within the zone \\ 
    \hline 
    \end{tabular} 
\end{table}
\begin{remark}
\textcolor{black}{
It is important to note that the value of \(T_{ori}^{out}\) depends on the environment adjacent to the \(ori\) wall. If the wall faces the building's exterior, \(T_{ori}^{out}\) is the outdoor ambient temperature, \(T_{out}\). However, if the wall is shared with an adjacent zone, \(T_{ori}^{out}\) is taken as the indoor air temperature of that neighboring zone. This mechanism ensures that inter-zone thermal coupling effects are inherently captured within each zone's RC model.}
\end{remark}

The expression for the indoor node is
\begin{equation}
	\begin{aligned}
		C_{z}^{} \frac{d T_{z}^{}}{d t}=&\frac{T_{n}^{w i}-T_{z}^{}}{R_{n}}+\frac{T_{e}^{w i}-T_{z}^{}}{R_{e}}+\frac{T_{w}^{w i}-T_{z}^{}}{R_{w}}\\
		&+\frac{T_{s}^{w i}-T_{z}^{}}{R_{s}}
		+u_{}+\dot{Q}_{z}^{i n}+\dot{Q}_{z}^{rad}
	\end{aligned}.
         \label{rc1}
\end{equation}
There are two types of wall nodes: inner surface nodes and outer surface nodes. The formula for inner and outer surface nodes of the wall is given by 
\begin{equation}
	C_{ {ori }}^{w} \frac{d T_{ {ori }}^{w i}}{d t}=\frac{T_{z}^{}-T_{ {ori }}^{w i}}{R_{ {ori }}}+\frac{T_{{ori }}^{w o}-T_{ {ori }}^{w i}}{R_{ {ori }}^{w}}
 \label{rc2}
\end{equation}
\begin{equation}
	C_{ {ori }}^{w} \frac{d T_{ {ori }}^{w o}}{d t}=\frac{T_{ {ori }}^{out}-T_{ {ori }}^{w o}}{R_{ {ori }}^{\prime}}+\frac{T_{ {ori }}^{w i}-T_{ {ori }}^{w o}}{R_{ {ori }}^{w}}+\dot{Q}_{ {ori }}^{ {rad }}.
 \label{rc3}
\end{equation}

For a zone with four walls, the indoor node is described by (\ref{rc1}), while the inner and outer surfaces of the walls are modeled by (\ref{rc2}) and (\ref{rc3}), respectively. The state variables include the indoor air temperature and the surface temperatures of the four walls, i.e.,
$
x=[\begin{array}{ll}
	T_{z} & T_{n}^{wi} 
\end{array}
$
$
\begin{array}{lllllll}
	T_{e}^{wi} & T_{w}^{wi} & T_{s}^{wi}& T_{n}^{wo} & T_{e}^{wo} & T_{w}^{wo} & T_{s}^{wo}
\end{array}]^\top
$
$
\in \mathbb{R}^{9}
$.
\textcolor{black}{The input $u\in \mathrm{R}$ is the input power of HVAC systems.} The RC model for a single zone includes nine differential equations and is expressed as follows:
\begin{equation}
		{x}^+ =\bar{A} x + \bar{B} u +  \bar{d},\\
	\label{state2}
\end{equation}
where $\bar{A}$, $\bar{B}$ and $\bar{d}$ are the corresponding constant coefficients.

Thus, the RC model for the $i$-th zone in a multi-zone system can be formulated and discretized as
\begin{equation}
		{x}_i(k+1) =A_i x_i(k) + B_i u_i(k) +  d_i(k),
	\label{state_m}
\end{equation}
where $A_i$, $B_i$, and $d_i$ are the corresponding matrices after discretisation.
The R-values and C-values used in this paper are listed in Table \ref{rc}. 

\begin{table}[htbp]
\vskip -0.1in
	\centering
	\caption{R-values and C-values}
	\label{rc}
    \scriptsize
	\begin{tabular}{cccc}
		\hline
		Parameter&Value($K/W$)&Parameter&Value($J/K$) \\
		\hline
		$R_{e,w}$ & 0.0232 & $C_{z}$ & $4.8\times 10^4$\\
		$R_{n,s}$ & 0.0310 & $C_{ {n}}^{w}$ & $8.5\times 10^5$\\
		$R_{e,w}^w$ & 0.0179 & $C_{ {e}}^{w}$ & $1.1\times 10^6$\\
		$R_{n,s}^w$ & 0.0238 & $C_{ {w}}^{w}$ & $1.1\times 10^6$\\
		$R_{e,w}^\prime$ & 0.0087 & $C_{ {s}}^{w}$ & $8.5\times 10^5$\\
		 $R_{n,s}^\prime$ & 0.0116 &&\\
		\hline
	\end{tabular}
    \vskip -0.2in
\end{table}

\subsection{Comfort index}
As a popular comfort indicator, the PMV index provides a comprehensive measure of indoor occupant comfort. In this paper, the PMV Index is used as one of the objective functions in the MPC scheme to enhance indoor comfort.
The expression for the PMV index is given by \cite{cigler2012optimization},
\begin{equation}
	\begin{aligned}
		P M V &=(0.303 \cdot \exp (-0.036 \cdot M)+0.028) \cdot L \\
		L &=(M-W)-3.05  \cdot[5.733-0.007 \cdot(M-W)\\
		&-p_{a}]-0.42 \cdot[(M-W)-58.15]-0.0173 \cdot M 
		\\\
		&\cdot (5.867-p_{a}) -0.0014 \cdot M \cdot\left(34-t_{a}\right)-3.96 \cdot 10^{-8} \\
		&\cdot f_{c l} \cdot[(t_{c l}+273)^{4}
		-\left(\bar{t}_{r}+273\right)^{4}]-f_{c l} \cdot h_{c} \cdot (t_{c l}-t_{a})\\
		t_{c l} &=35.7-0.0275 \cdot(M-W)-I_{c l} \cdot\{3 . 9 6 \cdot 1 0 ^ { - 8 } \cdot f _ { c l } \\
		&\cdot [(t_{c l}
		+273)^{4}-\left(\bar{t}_{r}+273\right)^{4}]+f_{c l} \cdot h_{c} \cdot\left(t_{c l}-t_{a}\right)\}\\
		h_{c} &=\max \left(2.38\cdot\left|{t}_{{cl}}-{t}_{{a}}\right|^{0.25}, 12.1\cdot \sqrt{v_{ar}}\right)\\
		f_{c l} &= \begin{cases}1.00+1.290 \cdot I_{c l} & \text { if } \: I_{c l} \leq 0.078 \\
			1.05+0.645 \cdot I_{c l} & \text { if } \: I_{c l}>0.078\end{cases}\\
		p_{a}&=\Phi \cdot 6.1094 \exp \left(\frac{17.625 t_{a}}{t_{a}+243.04}\right) \cdot 10^{-3}.
	\end{aligned}
	\label{PMV}
\end{equation}

The symbols and notations used in the PMV expression are shown in Table \ref{tab3}, where parameter values in this paper are also given according to the ASHRAE and ISO 7730 standards \cite{international2005ergonomics,ashrae2013standard}. 
The PMV scale ranges from -3 to 3, in which higher positive values indicate warmer conditions, lower negative values indicate cooler conditions, and values near 0 suggest a more comfortable environment.
Note that in this paper, $t_a = T_{z}$, and $\bar{t}_{r}$ represents the weighted sum of the inner surface wall temperatures in each direction. 

\begin{table*}[htbp]
\vskip -0.2in
	\centering
	\caption{Parameters defining thermal comfort}
	\label{tab3}
    \scriptsize
	\begin{tabular}{cccc}
		\hline
		Symbol&Parameter&Units& Values used in the case study\\
		\hline
		$M$&Metabolic rate&$W/m^2$&60 $W/m^2$ = 1 $met$, representing low activity levels (e.g., office work, staying at home)\\
		
		$W$&Mechanical energy&$W/m^2$& Considered negligible for low activity tasks\\

        $t_{a}$&Air temperature&$^{\circ}$C&
        Equal to the zone air temperature $T_z$ in the state vector\\

        $\bar{t}_{r}$&Mean radiant temperature&$^{\circ}$C&
        Computed as a weighted sum of inner wall surface temperatures\\
		
		$v_{a r}$&Relative air speed&$m/s$&0.1 $m/s$, typical for office settings\\
		
		$I_{c l}$&Clothing thermal resistance&$m^2K/W$& Set to 0.155 $m^2 K/W$ (1 $clo$) in heating seasons, and 0.5 $clo$ in warmer months for this study\\
		
		$f_{c l}$&Clothing area factor&-& Varies according to $I_{c l}$\\
		
		$\Phi$&Relative humidity&\%& Dependent on indoor conditions\\
		
		$p_{a}$&Water vapor partial pressure&$kPa$& Determined by $\Phi$ and $t_{a}$\\
		
		$h_{c}$&Convective heat transfer coefficient&$W/m^2K$& Influenced by $t_{a}$, $t_{cl}$, and $v_{a r}$\\
		
		$t_{cl}$&Clothing surface temperature&$^{\circ}$C& Affected by $M$, $W$, $I_{c l}$, $f_{c l}$, $\bar{t}_{r}$, and $h_{c}$\\
		\hline
	\end{tabular}
    \vskip -0.1in
\end{table*}

\textcolor{black}{With the prediction model and comfort index in the cost function now established, the next section will present the resulting MPC optimization problem and its distributed solution.}

\subsection{Explicit Mapping Between the PMV Index and the System State}

\textcolor{black}{
For clarity, we explicitly state the relationship between the PMV index and the
thermal state variables of each zone. In our formulation, the air temperature in
zone~$i$ corresponds to the state component $T_{z,i}$, while the mean radiant
temperature $\bar{t}_{r,i}$ is computed as a weighted sum of the inner surface
temperatures of the surrounding walls:
\[
    \bar{t}_{r,i}=\sum_{ori\in\{n,e,w,s\}} \omega_{ori} T^{wi}_{ori,i},
    \quad \sum_{ori} \omega_{ori}=1 .
\]
Thus, the PMV index can be written explicitly as
\[
PMV_i(k) = \Psi\big(T_{z,i}(k),\{T^{wi}_{ori,i}(k)\}_{ori}\big),
\]
where $\Psi(\cdot)$ denotes the nonlinear comfort model (\ref{PMV}) or its PWA
approximation (\ref{PMV_PWA}). Therefore, the PMV index is fully determined by
the system state $x_i(k)$ and does not introduce additional independent
decision variables.}

\section{PWA-based distributed MPC scheme}
\label{s3}
\label{distributed_nonlinear}
\textcolor{black}{In this section, we first present the large-scale problem, which includes the original nonlinear formulation of the PMV index. Then, we apply the PWA technique to approximate the PMV, transforming the large-scale problem into a PWA-based quadratic programming problem. Leveraging the properties of the PWA approximation, we further reformulate the problem as a large-scale convex optimization, which is then decoupled using ADMM. Finally, we introduce a convex ADMM algorithm to efficiently solve the PWA-based quadratic programming problem.}

\subsection{Nonlinear optimization problem formulation}

In this paper, the MPC aims to minimize two objectives: comfort cost and energy consumption. The economic objective function for a single subsystem is given by,
\begin{equation}
	 J^k_i=\alpha \cdot C^{\text{PMV}}_i+C^u_i,
	\label{Jsingle}
\end{equation}
where the superscript $i$ indicates the subsystem number, and $\alpha$ is a weight parameter that can be adjusted based on the occupants' preference (focus more on energy efficiency or comfort). $C^{\text{PMV}}_i$ and $C^u_i$ are defined as follows:
\begin{equation}
	C^{\text{PMV}}_i= \sum\limits_{l=1}^{N} \delta^{k}(l) ({PMV}_{i})^2,
\end{equation}
\begin{equation}
	C^u_i=\sum\limits_{l=1}^{N} \lambda^{k}(l) (u^{k+l}_{i})^2,
\end{equation}
where $N$ is the control horizon, and
${PMV}_{i}$ represents the original nonlinear formulation of the PMV index. The variable $\delta^{k}(l)$ indicates building occupancy, with a value of 1 when occupied and 0 otherwise.
The weight $\lambda^{k}(l)$ adjusts the importance of $C^u_i$ in the optimization problem, varying according to the electricity tariff as shown in Table \ref{fee}.
When the electricity tariff is high, $\lambda^{k}(l)$
encourages energy efficiency, while a low tariff favors comfort.
The term $u^{k+l}_i$ denotes the HVAC system’s input power.
\begin{table}[h]	
\vskip -0.1in
	\centering
	\caption{Electricity tariff.}
	\label{fee}
	\scriptsize
	\begin{tabular}{cc}
		\hline
		Time&Electricity tariff $\lambda^{k}(l)$ (CNY/kWh)\\
		\hline
		0:00-8:00&0.3358\\
		8:00-14:00&0.6629\\
		14:00-17:00&1.0881\\
		17:00-19:00&0.6629\\
		19:00-22:00&1.0881\\
		22:00-24:00&0.6629\\
		\hline				
	\end{tabular}
    \vskip -0.1in
\end{table}

The global objective function is described as
\begin{equation}
	\min\: \sum\limits_{i \in \mathcal{M}}J^k_i,
	\label{Jtotal}
\end{equation}
where $\mathcal{M}=\{1, 2, \cdots, M\}$. Note that the global system considered in this paper comprises $M$ subsystems.
Then the MPC optimization problem can be described as
\begin{subequations}
	\begin{align}
\mathcal{P}_1:&\min\limits_{\mathbf{u}_1,\mathbf{u}_2,\cdots,\mathbf{u}_M}\: \sum\limits_{i \in \mathcal{M}}J^k_i=\sum\limits_{i \in \mathcal{M}}\bigg(\alpha \cdot C^{\text{PMV}}_i+C^u_i\bigg) \\
		&\quad\:\:\text { s.t.\: } \quad\quad0\leq \sum\limits_{i\in\mathcal{M}}\mathbf{u}_i \leq \mathbf{c}^{\max}\label{uuuuu}\\
		& \quad\quad\quad\quad\quad\quad \forall i \in \mathcal{M}:\nonumber \\
		& \quad\quad\quad\quad\quad\begin{aligned}[t]
			x_i(k+1)=&A_ix_i(k)+B_{i} u_i(k)+  d_i(k) 
		\end{aligned}
		\label{stateequation}\\
		&\quad\quad\quad\quad\quad x_i(k)=x^0_{i} \label{intial}\\
		&\quad\quad\quad\quad\quad\mathbf{u}^{\min }_i \leq \mathbf{u}_i \leq \mathbf{u}^{\max }_i \label{ulimit}.
	\end{align}
	\label{PMV_l0}
\end{subequations}
Here, $\mathbf{c}^{\max}$ denotes the maximum power supplied by the energy system, $x^0_{i}$ represents the initial state at the $k$-th moment, and $\mathbf{u}^{\min }_i$ and $\mathbf{u}^{\max }_i$ correspond to the minimum and maximum input power, respectively.
Note that $\mathcal{P}_1$ is a large-scale nonconvex problem.

\subsection{PWA-based optimization problem formulation}
\textcolor{black}{The inherent nonlinearity of the PMV index limits computational efficiency for solving $\mathcal{P}_1$, and even small-scale problems derived from decomposing $\mathcal{P}_1$ using distributed algorithms like ADMM continue to exhibit nonconvexity, posing challenges for optimization.}
To address this issue, the methodology begins by employing a PWA approximation of the PMV index, based on the technique presented in \cite{li2023economic}. This approach replaces the original nonlinear function with a set of affine functions, each valid within a distinct operating subregion. This step eliminates the nonlinear component in $\mathcal{P}_1$. Although the resulting optimization problem $\mathcal{P}_2$ remains nonconvex, the PWA structure is further leveraged to reformulate $\mathcal{P}_2$ as a convex problem, $\mathcal{P}_3$.

Let $\widehat{PMV}$ represent the PWA approximation of the PMV index, and (\ref{PMV}) is reduced to:
\begin{equation}
	\begin{aligned}
		&\hat{t}_{cl}=f_{\text {PWA }}\left(t_{a}, \bar{t}_{r}\right) \\
		&\hat{p}_{a}=\Phi \cdot\left(1.7833 \cdot t_{a}-12.7516\right) \cdot 10^{-3}\\
		&\widehat{PMV}=0.2551 \cdot \hat{p}_{a} + 0.0052 \cdot t_{a}\\
		&\qquad\qquad\:\:\:+0.8052 \cdot \hat{t}_{cl}-25.2883.
	\end{aligned}
	\label{PMV_PWA}
\end{equation}
\textcolor{black}{
To clarify the construction of the PWA model in \eqref{PMV_PWA}, we briefly
describe the approximation procedure used in this paper. The nonlinearity of
the PMV index is mainly driven by the air temperature $t_a$ and the mean
radiant temperature $\bar{t}_r$. Therefore, we construct a two-dimensional PWA
representation in the $(t_a,\bar{t}_r)$ space and approximate the clothing
surface temperature $t_{cl}$ by a piecewise-affine function.}

\textcolor{black}{
Specifically, we consider the typical comfort-oriented operating range for the
case study and select four linearization points,
\[
(24,24),\ (24,28),\ (28,24),\ (28,28)\ (^{\circ}\mathrm{C}),
\]
in the $(t_a,\bar{t}_r)$ plane. These points form a regular $2\times 2$ grid
and partition the operating domain into four regions. In each region, the
clothing surface temperature is approximated by a local affine function of the
form
\[
  \hat{t}_{cl} \approx a_1 t_a + a_2 \bar{t}_r + a_3,
\]
where the coefficients $a_1$, $a_2$, and $a_3$ are obtained offline such that
the resulting PWA model is continuous across the region boundaries. Substituting
this affine expression into the PMV formula (\ref{PMV}) yields, for each region,
a local affine approximation of the PMV index. The collection of these local
models gives the global PWA representation $\widehat{PMV}$ used in
\eqref{PMV_PWA}, which is well suited for optimization-based control.}

\textcolor{black}{
The approximation accuracy of the PWA model is evaluated against the full
nonlinear PMV model on a uniformly sampled grid over the considered
temperature range. Let $PMV(t_a,\bar{t}_r)$ denote the nonlinear index and
$\widehat{PMV}(t_a,\bar{t}_r)$ its PWA approximation. The mean absolute error
(MAE) is defined as
\[
  \mathrm{MAE}
  = \frac{1}{N_g}\sum_{q=1}^{N_g}
    \big| PMV(t_a^{(q)},\bar{t}_r^{(q)})
       - \widehat{PMV}(t_a^{(q)},\bar{t}_r^{(q)}) \big|.
\]
In the numerical evaluation, $t_a$ and $\bar{t}_r$ are sampled on a $20\times20$
uniform grid over $[22,30]\,^{\circ}\mathrm{C}$, i.e., $N_g=400$ grid points,
and all other parameters are fixed to the values given in Table~\ref{tab3}.
In our implementation, the resulting error is
\[
  \mathrm{MAE} = 0.0099,
\]
which is well within the commonly accepted tolerance for thermal comfort
modeling (typically within $\pm 0.1$ PMV). This indicates that the PWA
approximation is sufficiently accurate for the MPC-based comfort and energy
optimization considered in this paper.}

The comfort objective is then rewritten as:
\begin{equation}
	\hat{C}^{\text{PMV}}_i= \sum\limits_{l=1}^{N} \delta^{k}(l) (\widehat{PMV}_{i})^2,
\end{equation}
and the MPC optimization problem is reformulated as:
\begin{equation}
	\begin{aligned}		\mathcal{P}_2:&\min\limits_{\mathbf{u}_1,\mathbf{u}_2,\cdots,\mathbf{u}_M}\: &&\sum\limits_{i \in \mathcal{M}}\hat{J}^k_i=\sum\limits_{i \in \mathcal{M}}\bigg(\alpha \cdot \hat{C}^{\text{PMV}}_i+C^u_i\bigg) \\
		&\quad\:\:\text { s.t.\: }&& (\ref{uuuuu})\\
		 &&& \forall i \in \mathcal{M}:\: (\ref{stateequation})-(\ref{ulimit})	.
	\end{aligned}
	\label{PMV_l}
\end{equation}
Note that $\mathcal{P}_2$ is a nonconvex piecewise quadratic programming problem. 
\textcolor{black}{
To solve $\mathcal{P}_2$, the properties of PWA are utilized to derive an affine expression for the PMV in the active subregion. This leads to a series of convex problems, i.e., the following problem $\mathcal{P}_3$.}

Denote $\widetilde{PMV}_{i, s}$ as the affine function of $\widehat{PMV}_{i}$ after activation by the specified $\mathbf{u}_{i,s}$, i.e., for subsystem $i$, the $s$-th subregion is activated. The quadratic function $\tilde{C}^{\text{PMV}}_{i,s}$ is
\begin{equation}
	\tilde{C}^{\text{PMV}}_{i,s}= \sum\limits_{l=1}^{N} \delta^{k}(l) (\widetilde{PMV}_{i,s})^{2}.
 \label{affine_ex}
\end{equation}
Then a resulting quadratic programming problem is given by:
\begin{equation}
	\begin{aligned}		\mathcal{P}_3:&\min\limits_{\mathbf{u}_1,\mathbf{u}_2,\cdots,\mathbf{u}_M}\: &&\sum\limits_{i \in \mathcal{M}}\tilde{J}^k_{i,s}=\sum\limits_{i \in \mathcal{M}}\bigg(\alpha \cdot \tilde{C}^{\text{PMV}}_{i,s}+C^u_i\bigg) \\
		&\quad\:\:\text { s.t.\: }&& (\ref{uuuuu})\\
		 &&& \forall i \in \mathcal{M}:\: (\ref{stateequation})-(\ref{ulimit})	.
	\end{aligned}
	\label{PMV_l1}
\end{equation}
Note that $\mathcal{P}_3$ is convex.

\begin{remark}
%\textcolor{black}{\textcolor{black}%{For the} PWA approximation, the feasible domain of $\mathbf{u}_i$ is divided into several non-overlapping subregions, each associated with a distinct affine function.}
\textcolor{black}{In the optimization problem $\mathcal{P}_3$, the constraints of activated PWA subregions are not included. Yet a local optimum of $\mathcal{P}_2$ can still be achieved by solving a series of $\mathcal{P}_3$, during which different PWA subregions are traversed. See the detailed analysis in Section \ref{sec:covexadmm}.}
\end{remark}

\subsection{Effect of PWA Approximation on Comfort Performance}
\textcolor{black}{
In the proposed framework, the PWA approximation is applied only to the comfort
term associated with the PMV index, while the system dynamics
(\ref{stateequation}), the input bounds (\ref{ulimit}), and the coupling
constraint (\ref{uuuuu}) are kept exact. Therefore, the feasibility of the MPC
problem is not affected by the PWA approximation: any feasible solution for the
nonlinear formulation $\mathcal{P}_1$ remains feasible for the PWA-based
formulation $\mathcal{P}_2$.}

\textcolor{black}{
To analyze the impact on comfort performance, we denote the true PMV index by
$PMV_i$ and its PWA approximation by $\widehat{PMV}_i$. For subsystem $i$, the
comfort cost in the original formulation is
\[
  C^{\text{PMV}}_i = \sum_{l=1}^{N} \delta^{k}(l)\big(PMV_i\big)^2,
\]
while in the PWA-based formulation it becomes
\[
  \widehat{C}^{\text{PMV}}_i = \sum_{l=1}^{N} \delta^{k}(l)\big(\widehat{PMV}_i\big)^2.
\]
Define the approximation error
\[
  e_{\mathrm{PWA},i}(l) = PMV_i(l) - \widehat{PMV}_i(l).
\]
Then the difference between the two comfort costs can be written as
\begin{align*}
  C^{\text{PMV}}_i - \widehat{C}^{\text{PMV}}_i
  &= \sum_{l=1}^{N} \delta^{k}(l)\Big(PMV_i^2(l) - \widehat{PMV}_i^2(l)\Big) \\
  &= \sum_{l=1}^{N} \delta^{k}(l)\, e_{\mathrm{PWA},i}(l)\,
     \big(PMV_i(l) + \widehat{PMV}_i(l)\big).
\end{align*}
Suppose that, over the operating range of interest, both the true and
approximated PMV values are confined to a compact interval
\[
  |PMV_i(l)| \leq \bar{p}, \quad |\widehat{PMV}_i(l)| \leq \bar{p}, \quad
  \forall l,
\]
which is consistent with the comfort-oriented MPC design. If the PWA
approximation error is bounded by
\[
  |e_{\mathrm{PWA},i}(l)| \leq \varepsilon_{\max}, \quad \forall l,
\]
then the deviation between the two comfort costs satisfies
\[
  \big|C^{\text{PMV}}_i - \widehat{C}^{\text{PMV}}_i\big|
  \leq 2 N \bar{p}\,\varepsilon_{\max}.
\]
Consequently, the difference between the corresponding local objectives
$J_i^k = \alpha C^{\text{PMV}}_i + C^u_i$ and
$\widehat{J}_i^k = \alpha \widehat{C}^{\text{PMV}}_i + C^u_i$ is bounded by a
constant that is proportional to the worst-case PWA approximation error
$\varepsilon_{\max}$. This indicates that, as long as the PWA model provides a
reasonable approximation of the nonlinear PMV index in the comfort region, the degradation in comfort performance is theoretically limited.}

\subsection{ADMM iteration}
For the large-scale problem $\mathcal{P}_3$, centralized algorithms face scalability challenges as the number of subsystems increases. Distributed optimization algorithms, such as ADMM, are more suitable for addressing these large-scale problems. 

Define
$\mathcal{U}_i=\{\mathbf{u}_i|(\ref{stateequation})-(\ref{ulimit}), \forall i \in \mathcal{M}\}$ as the local constraint set.
$\mathcal{P}_3$ can be reformulated as

\begin{equation}
	\begin{aligned}
		&\min\limits_{\mathbf{u}_i\in\mathcal{U}_i,i\in\mathcal{M}}\: &&\sum\limits_{i \in \mathcal{M}}\tilde{J}^k_{i,s} \\
		&\quad\:\:\text { s.t.\: } && 0 \leq \sum\limits_{i\in\mathcal{M}}\mathbf{u}_i \leq \mathbf{c}^{\max}.
	\end{aligned}
	\label{PMV_5}
\end{equation}
We introduce an indicator function $\mathrm{I}_S$ and an auxiliary variable $\mathbf{z}$. $\mathcal{P}_3$ is then reformulated as
\begin{equation}
	\begin{aligned}
		&\min\limits_{\mathbf{u}_i\in\mathcal{U}_i,i\in\mathcal{M}} && \sum\limits_{i \in \mathcal{M}}\tilde{J}^k_{i,s}+\mathrm{I}_S(\mathbf{z}) \\
		&\quad\:\:\text { s.t. } && \sum\limits_{i\in\mathcal{M}}\mathbf{u}_i=\mathbf{z} \\
		&&& \mathrm{I}_S(\mathbf{z})= \begin{cases}0, & \mathbf{z} \in S=\left[0, \mathbf{c}^{\max }\right] \\
			\infty, & \mathbf{z} \in \text { other }.\end{cases}
	\end{aligned}
	\label{qp}
\end{equation}
The Lagrangian for problem (\ref{qp}) is
\begin{equation}
	\begin{aligned}
		& \mathbf{L}_{\rho}(\mathbf{u},\mathbf{z},\bm{\lambda})=\sum\limits_{i \in \mathcal{M}}\bigg(\alpha \cdot \tilde{C}^{\text{PMV}}_{i,s}+C^u_i\bigg)+\mathrm{I}_S(\mathbf{z})\\
		&\quad\quad\quad\quad\quad+\bm{\lambda}{ }^\top \bigg(\sum\limits_{i \in \mathcal{M}}\mathbf{u}_i-\mathbf{z}\bigg)\nonumber+\frac{\rho}{2} \left\|\sum\limits_{i \in \mathcal{M}}\mathbf{u}_i-\mathbf{z}\right\|_2^2.
	\end{aligned}
\end{equation}
The problem is now tractable, and we have the following ADMM iterations:
\begin{subequations}
	\begin{flalign}
	& \mathbf{u}^{\tau+1}_i=\arg \min _{\mathbf{u}_i \in \mathcal{U}_i}\Bigg\{\tilde{C}^{\text{PMV}}_{i,s}+C^u_i+\nonumber\\
 &\quad\quad\quad\frac{\rho}{2}\left\|\mathbf{u}_i+\sum\limits_{j\in\mathcal{M}}^{j\neq i} \mathbf{u}^\tau_j-\mathbf{z}^\tau+\frac{\bm{\lambda}^\tau}{\rho}\right\|_2^2 \Bigg\}
 \label{admmy}\\
	&\begin{aligned}[t] &\mathbf{z}^{\tau+1}=\prod\nolimits_S\bigg(\sum_{i \in \mathcal{M}} \mathbf{u}^{\tau}_i+\frac{\bm{\lambda}^\tau}{\rho}\bigg)  \end{aligned} \label{10b}\\
	& \bm{\lambda}^{\tau+1}=\bm{\lambda}^\tau+\rho\bigg(\sum_{i \in \mathcal{M}} \mathbf{u}^{\tau+1}_i-\mathbf{z}^{\tau+1}\bigg),\label{10c}
	\end{flalign}
 \label{ADMM}
\end{subequations}
where $\prod\nolimits_S$ is the Euclidean projection onto the set $S$.
\begin{remark}
 During the ADMM iterations, the updates for $\mathbf{z}$ and $\bm{\lambda}$ follow directly from the derivation above, leaving only the update of $\mathbf{u}_i$ requiring the solution of an optimization problem. As ADMM decouples the global optimization problem, the update of $\mathbf{u}_i$ depends on the respective subsystems. This allows all $M$ subsystems to update their own $\mathbf{u}_i$ in parallel.
\end{remark}

Next, we give Lemma \ref{lemma1} to show that ADMM iterations (\ref{ADMM}) converge to an optimum of the convex problem $\mathcal{P}_3$.
\begin{lemma}
    The iterates $\{\mathbf{u}^{\tau}_1,\cdots,\mathbf{u}^{\tau}_M,\mathbf{z}^{\tau}
,\bm{\lambda}^{\tau}\}$ from the ADMM algorithm (\ref{ADMM}) converge linearly to the optimal primal-dual solution of $\mathcal{P}_3$.
\label{lemma1}
\end{lemma}
\begin{proof}
The updates of $\mathbf{u}^{\tau+1}_i$ and $\mathbf{z}^{\tau+1}$ are based on the $\tau$-th iterates $\mathbf{u}^{\tau}_i$, $\mathbf{z}^{\tau}$, and $\bm{\lambda}^\tau$, while the updates of $\bm{\lambda}^{\tau+1}$ are based on the $(\tau+1)$-th iterates of $\mathbf{u}^{\tau+1}_i$ and $\mathbf{z}^{\tau+1}$. This implies the ADMM variables are updated in a Jacobi manner. \cite{hong2017linear} proves that the Jacobi update leads to the ADMM algorithm converging linearly to the optimal primal-dual solution for convex problems with shared constraints. Therefore, the iterates $\{\mathbf{u}^{\tau}_1, \cdots, \mathbf{u}^{\tau}_M, \mathbf{z}^{\tau}, \bm{\lambda}^{\tau}\}$ in (\ref{ADMM}) converge linearly to the optimal primal-dual solution of $\mathcal{P}_3$.
\end{proof}

Here the update of ADMM variables has a simpler formulation, i.e.,
\begin{subequations}
	\begin{flalign}
	& \mathbf{u}^{\tau+1}_i = \arg \min _{\mathbf{u}_i \in \mathcal{U}_i}\Bigg\{\tilde{C}^{\text{PMV}}_i + C^u_i + \nonumber\\
	&\quad\quad\quad \frac{\rho}{2} \left\| \mathbf{u}_i + \sum\limits_{j\in\mathcal{M}}^{ j\neq i} \mathbf{u}^\tau_j - \mathbf{z}^\tau + \bm{\theta}^{\tau} \right\|_2^2 \Bigg\}
	\\
	& \mathbf{z}^{\tau+1} = \prod\nolimits_S \left( \sum_{i \in \mathcal{M}} \mathbf{u}^{\tau}_i + \bm{\theta}^{\tau} \right)
	\\
	& \bm{\theta}^{\tau+1} = \bm{\theta}^\tau + \sum_{i \in \mathcal{M}} \mathbf{u}^{\tau+1}_i - \mathbf{z}^{\tau+1},
	\end{flalign}
\end{subequations}
if using a scaled dual variable $\bm{\theta}=\bm{\lambda}\text{/} \rho$. See \cite{boyd2011distributed} for reference.

\textcolor{black}{
The difference in solution updates in ADMM refers to the change in the solution between consecutive iterations. It is defined as:
$$r=\frac{1}{M}\sum_{i \in \mathcal{M}} \|\mathbf{u}^{\tau+1}_i-\mathbf{u}^{\tau}_i\|_1,$$ 
where $\mathbf{u}^{\tau}$ and $\mathbf{u}^{\tau+1}$ denote the values of the solution at the $\tau$-th and $\tau+1$-th iterations, respectively, and $\|\cdot\|_1$ denotes the $L_1$-norm. }

\textcolor{black}{
Figure \ref{ADMM_res} illustrates the difference in solution updates for the problem $\mathcal{P}_3$ solved by (\ref{ADMM}).
Here, the parameter values of PMV are obtained according to Table \ref{tab3}, with the control horizon set to $N=12$ and the number of subsystems set to $M=36$.
As the number of iterations increases, the curve shows a rapid convergence to zero. This indicates that the solution updates become smaller with each iteration and that the algorithm is approaching a stable solution, i.e., the optimal solution of $\mathcal{P}_3$.}
\begin{figure}[htbp]
\vskip -0.1in
	\centering
	\includegraphics[width=0.35\textwidth]{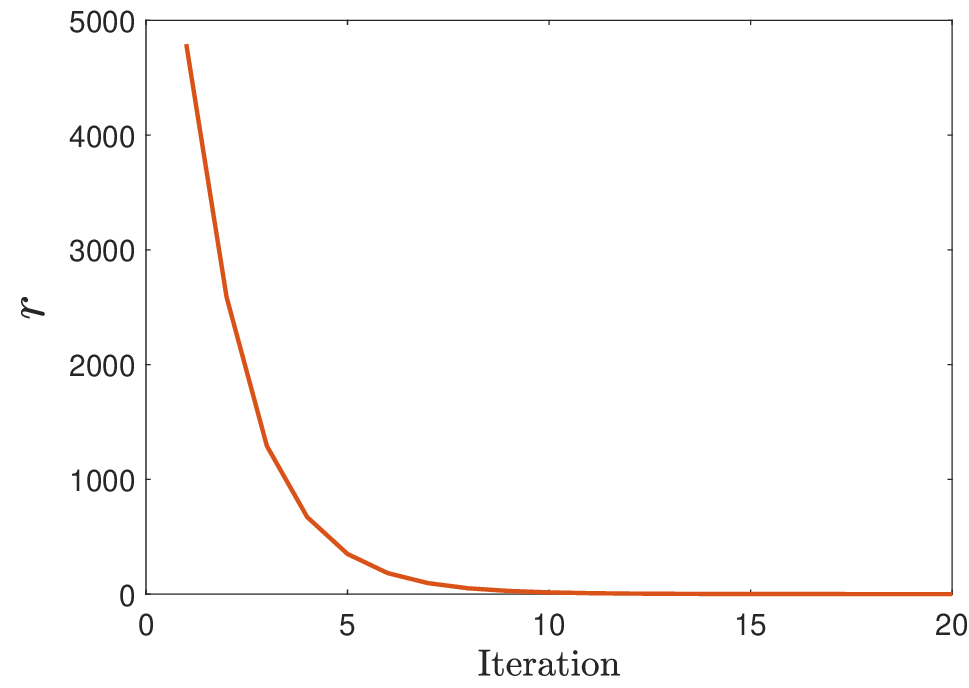} 
	\caption{Difference in solution updates.} 
	\label{ADMM_res} 
    \vskip -0.3in
\end{figure}

\subsection{Convex ADMM algorithm}\label{sec:covexadmm}
% Version A (formal, fluent)
In this section, we present a novel algorithm for the nonconvex problem \(\mathcal{P}_2\). As noted earlier, on any fixed PWA subregion, \(\mathcal{P}_2\) reduces to the convex problem \(\mathcal{P}_3\), which we solve iteratively via ADMM. By exploiting this reformulation, the algorithm converges to a local optimum of \(\mathcal{P}_2\).

To solve the nonconvex problem $\mathcal{P}_2$, we introduce Algorithm \ref{explore}, which operates through a sequence of activation subregions. Given a control sequence, an affine expression for the PMV index is derived based on the current activation subregion, and this leads to the convex problem $\mathcal{P}_3$. 
\textcolor{black}{
Since $\mathcal{P}_3$ does not include a PWA subregion constraint, there are two possibilities for the solution of this quadratic programming problem. The first is that the solution lies within the currently activated PWA subregion, which is desirable as it is locally optimal for $\mathcal{P}_2$. The second is that the solution lies outside the current PWA subregion, indicating that the subregion likely has no local optimum for $\mathcal{P}_2$. However, this situation allows the solution to escape the current subregion, activating other subregions and leading to a new $\mathcal{P}_3$, from which we can iteratively search for a local optimum for $\mathcal{P}_2$.}
\textcolor{black}{Algorithm \ref{explore} shows these two possibilities.}

\begin{algorithm}[htbp]
	\caption{Convex-ADMM for solving $\mathcal{P}_2$.}
	\label{explore}
	\begin{algorithmic}[1]
		\Require  The initial control sequences $\{\mathbf{u}^0_1, \mathbf{u}^0_2, \cdots, \mathbf{u}^0_M\}$, the initial states $\{x^0_{1}, x^0_{2}, \cdots, x^0_{M}\}$.
		\Ensure The solution $\{\mathbf{u}^{*}_1, \mathbf{u}^{*}_2, \cdots, \mathbf{u}^{*}_M\}$.
		\For {$\tau=0,1,\cdots,T-1$}
        \State \quad$\forall i \in \mathcal{M}:$
        \If {$\tau<T_{d}$}
        \State Get the new activated subregion within the control
        \Statex \quad \quad \quad horizon according to $\mathbf{u}^{\tau}_i$ and reset (\ref{affine_ex}).
        \EndIf
        \State Get $\mathbf{u}^{\tau+1}_i$ from (\ref{admmy}).
        \State Get $\mathbf{z}^{\tau+1}_i$ from (\ref{10b}).
        \State Get $\bm{\lambda}^{\tau+1}_i$ from (\ref{10c}).
        \EndFor
        \State Get $ \{\mathbf{u}_{1,s_1}, \mathbf{u}_{2,s_2}, \cdots, \mathbf{u}_{M,s_M}\}$ and the current activated subregions $\{\bm{\Omega}_{1,s_1},\bm{\Omega}_{2,s_2},\cdots,\bm{\Omega}_{M,s_M}\}$.
        \If {$\mathbf{u}_{i,s_i} \notin \text{int}(\bm{\Omega}_{i,s_i})$}
        \State Given a new initial $\mathbf{u}^0_i$ and return to Step 1.
        \EndIf
        \State $ \{\mathbf{u}^{*}_1, \mathbf{u}^{*}_2, \cdots, \mathbf{u}^{*}_M\}=  \{\mathbf{u}_{1,s_1}, \mathbf{u}_{2,s_2}, \cdots, \mathbf{u}_{M,s_M}\}$ is determined.
	\end{algorithmic}
\end{algorithm}

Next, we present a theorem to demonstrate that our algorithm can achieve a local optimum for $\mathcal{P}_2$.

\begin{assumption}
Assume for each subsystem $i$, there is a local optimum $\mathbf{u}^{*}_{i,s} $ of $\mathcal{P}_2$ such that $\mathbf{u}^{*}_{i,s} $ lies in the interior of some subregion $\bm{\Omega}_{i,s}$, i.e., $\mathbf{u}^{*}_{i,s} \in \text{int}(\bm{\Omega}_{i,s})$, $\forall i\in \mathcal{M}.$
\label{ass1}
\end{assumption}

\begin{remark}
\textcolor{black}{
Assumption 1 serves as a simplifying condition for theoretical tractability, thus constituting a limitation of our current theoretical guarantee. In practice, the true optimal solution for a PWA problem might precisely reside on a subregion boundary, where the objective function is non-differentiable. In such cases, conventional convex optimization methods, relying on an interior solution assumption, might not directly capture a boundary optimum, potentially leading the algorithm to converge to a sub-optimal solution. Nevertheless, in many practical numerical optimization scenarios, solutions rarely fall exactly on strict mathematical boundaries due to precision limits and model perturbations, which lessens the restrictiveness of Assumption 1 at a numerical level. Furthermore, our algorithm's exploratory mechanism (especially during the \(T_d\) iterations with frequent active subregion switching) also contributes to effectively exploring the solution space in practice, helping to identify promising local optima, even if a rigorous guarantee for finding a boundary optimum is not provided.}
\end{remark}

\begin{theorem}
Under Assumption~\ref{ass1}, given an initial sequence $\{\mathbf{u}^0_1, \mathbf{u}^0_2, \ldots, \mathbf{u}^0_M\}$ and initial states $\{x^0_{1}, x^0_{2}, \ldots, x^0_{M}\}$, Algorithm~\ref{explore} is guaranteed to converge to a local optimum of $\mathcal{P}_2$.
\end{theorem}

\begin{proof}
    The subregions $\{\bm{\Omega}_{1,s_1},\bm{\Omega}_{2,s_2},\cdots,\bm{\Omega}_{M,s_M}\}$ are determined by (\ref{affine_ex}) when $\tau\geq T_d$.
    As \textcolor{black}{the number of iterations $T\rightarrow \infty$},
    $\{\mathbf{u}^{*}_1, \mathbf{u}^{*}_2, \cdots, \mathbf{u}^{*}_M\}$ is optimal for $\mathcal{P}_3$ according to Lemma \ref{lemma1}, and the solution is also optimal for the following problem, 
    \begin{equation}
	\begin{aligned}		\mathcal{P}_4:&\min\limits_{\mathbf{u}_1,\mathbf{u}_2,\cdots,\mathbf{u}_M}\: &&\sum\limits_{i \in \mathcal{M}}\tilde{J}^k_{i,s_i}=\sum\limits_{i \in \mathcal{M}}\bigg(\alpha \cdot \tilde{C}^{\text{PMV}}_{i,s_i}+C^u_i\bigg) \\
		&\quad\:\:\text { s.t.\: }&& (\ref{uuuuu})\\
		 &&& \quad \forall i \in \mathcal{M}:\: \\
   &&&(\ref{stateequation})-(\ref{ulimit}), \mathbf{u}_i \in {\bm{\Omega}}_{i,s_i} 	.
	\end{aligned}
	\label{PMV_l2}
    \end{equation}
    \textcolor{black}{Note that $\mathcal{P}_4$ includes the constraints of the activated PWA subregions and that $\mathbf{u}^{*}_i \in \text{int}(\bm{\Omega}_{i,s_i})$ holds for all $ i\in\mathcal{M}$ in Algorithm \ref{explore}}. Consequently, we conclude that $\{\mathbf{u}^{*}_1, \mathbf{u}^{*}_2, \cdots, \mathbf{u}^{*}_M\}$ is optimal for $\mathcal{P}_4$ within the neighborhood. 
    Therefore, $\{\mathbf{u}^{*}_1, \mathbf{u}^{*}_2, \cdots, \mathbf{u}^{*}_M\}$ is locally optimal for $\mathcal{P}_2$.
\end{proof}

It is important to note that although we use $\mathbf{u}_0^i$ in a specific PWA subregion to formulate the quadratic programming problem $\mathcal{P}_3$, the solution set is not restricted to this PWA subregion.  Instead, it can traverse across different PWA subregions until a local optimum of $\mathcal{P}_2$ is reached.

\subsection{Initialization Strategy and Practical Behavior of Algorithm~\ref{explore}}
\textcolor{black}{
In the closed-loop implementation, the proposed distributed MPC scheme is
solved at each sampling instant in a receding-horizon fashion. A natural and
effective choice for the initial control sequences
$\{\mathbf{u}^0_1,\mathbf{u}^0_2,\dots,\mathbf{u}^0_M\}$ at time step $k$ is to
reuse the optimal sequences obtained at the previous time step $k-1$, shifted
by one step forward and complemented with a feasible terminal input. More
specifically, let $\{\mathbf{u}^{*,k-1}_1,\dots,\mathbf{u}^{*,k-1}_M\}$ denote
the optimal solutions of the MPC problem at time $k-1$. At time $k$, the
initial sequence for subsystem $i$ is set as
\[
  \mathbf{u}^0_i =
  \big[u^{*,k-1}_i(2),\dots,u^{*,k-1}_i(N),u^{\mathrm{term}}_i\big],
  \quad \forall i\in\mathcal{M},
\]
where $u^{\mathrm{term}}_i$ is any feasible input that satisfies the local
constraints of subsystem $i$. This warm-start strategy is standard in MPC and
takes advantage of the temporal coherence of the optimal control profiles.}

\textcolor{black}{
This initialization has two implications for Algorithm~\ref{explore}. First, it
provides a feasible starting point that is typically close to the new optimum,
thereby accelerating the convergence of the ADMM iterations. Second, it
increases the practical relevance of Assumption~\ref{ass1}: since the optimal
inputs at consecutive time steps are usually close, the active PWA subregions
tend to remain unchanged or change smoothly, so that the local optimum often
lies in the interior of the same subregion over successive MPC steps.}

\textcolor{black}{
It is also worth noting that Assumption~\ref{ass1} is only required for the
theoretical convergence analysis. Algorithm~\ref{explore} itself remains valid
even when the assumption does not hold exactly. When the current iterate leaves
the activated PWA subregion, the algorithm simply updates the active subregion,
recomputes the affine expression in (\ref{affine_ex}), and continues the ADMM
iterations. }

\section{Case study}
\label{s5}
\subsection{Large-scale setup}
The case study in this paper is based on \emph{summer} weather conditions in \emph{Beijing}, China.
The large-scale system under consideration is shown in Fig. \ref{Untitled2}. It comprises 36 zones within a 9-floor building, where each floor has an identical layout. The $x$-th floor is divided into four zones, labeled $x$01, $x$02, $x$03, and $x$04, where $x$ is an integer ranging from 1 to 9. 
Each zone is of equal area and equipped with an ideal variable air volume terminal. This setup allows the supply airflow to be adjusted from zero to the maximum to meet the zone's heating or cooling demand. 
Key zone parameters are detailed in Table \ref{tab1}.
Note that the occupied hours are from 10:00 to 20:00.
The test is performed on a computing device equipped with an AMD Ryzen 7 5800H processor, operating at 3.20 GHz.
\begin{figure}[htbp]
\vskip -0.1in
	\centering
	\includegraphics[width=0.16\textwidth]{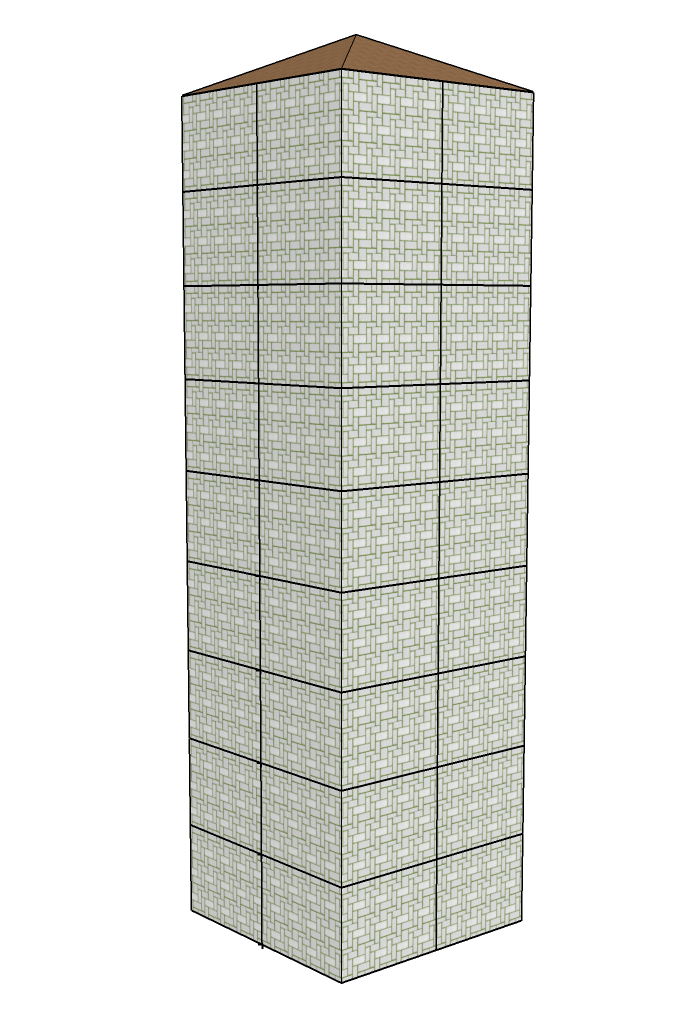} 
	\caption{ The considered building.} 
	\label{Untitled2} 
    \vskip -0.2in
\end{figure}

\subsection{Test results}
The control horizon for all MPC tests is set to $N=12$, and the ADMM parameter $\rho$ is set to 0.1.
\textcolor{black}{
The physical time interval corresponding to each control step is 15 minutes.}
We first present the test results for the 36-zone system over a 24-hour period. Fig. \ref{heatmap} illustrates the temperature heatmap for 36 zones. The trend of temperature variations in these 36 zones is consistent over a 24-hour period, and the indoor temperature remains within a comfortable range throughout the occupancy period. Due to the predictive capability of the MPC, the temperature begins to drop before 10:00, ensuring comfort by the start of occupancy. In addition, the indoor temperature shows a trend of low-high-low-high during the occupancy period, which corresponds to fluctuations in electricity tariffs. It indicates that when the electricity tariff is high, the controller adjusts by sacrificing some comfort to enhance cost efficiency.
\begin{table}[t]
\vskip -0.1in
	\centering
	\caption{Key zone parameters}
	\label{tab1}
    \scriptsize
	\begin{tabular}{cc}
		\hline
		Parameters&Preferences \\
		\hline
        Zone area& 16m$^2$ \\
		Occupant&1 occupant/12m$^2$ \\
		Lighting&0.75watts/m$^2$ \\
		Equipment&0.4watts/m$^2$ \\
		Occupied hours&10:00-20:00 \\
		\hline
	\end{tabular}
    \vskip -0.2in
\end{table}
\begin{figure*}[htbp]
  \vskip -0.1in  
	\centering
	\includegraphics[width=0.8\textwidth]{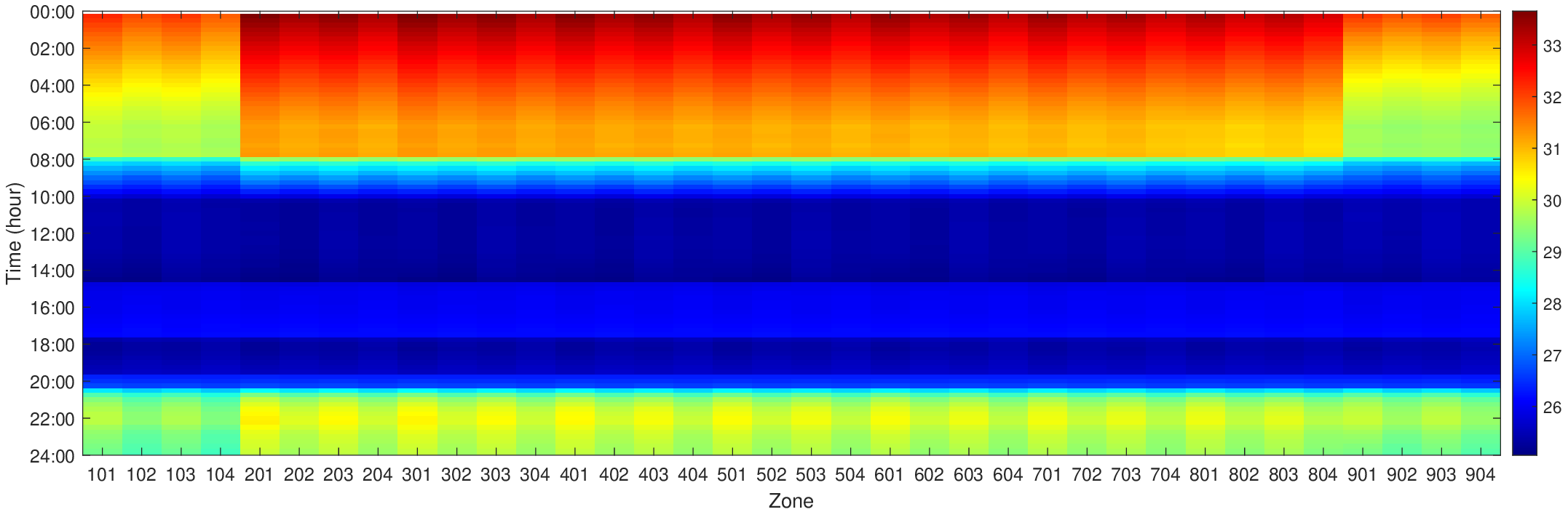} 
	\caption{Indoor temperature heatmap for 36 zones over 24 hours. Note that the occupied hours are from 10:00 to 20:00.} 
	\label{heatmap} 
  \vskip -0.1in  
\end{figure*}

Fig. \ref{PMV_distributed} illustrates the change in PMV for 36 zones. Throughout the occupancy period, the PMV values are closer to 0, indicating a higher level of comfort in the zone. An increase in PMV is observed during peak tariff periods, similar to the trend in indoor temperature. Additionally, Fig. \ref{Power_distributed} illustrates the power variations in each zone over a 24-hour period. The trade-off between energy and comfort is clearly illustrated by combining Figs. \ref{PMV_distributed} and \ref{Power_distributed}. During the 10:00-20:00 occupancy period, the demand for indoor comfort results in increased energy consumption, as reflected by the decreasing PMV curves in Fig. \ref{PMV_distributed} and the rising power curves in Fig. \ref{Power_distributed}. Outside of these hours, the PMV curve rises while the power curve approaches zero due to no comfort requirement.

\begin{figure}[htbp]
	\centering
	\includegraphics[width=0.48\textwidth]{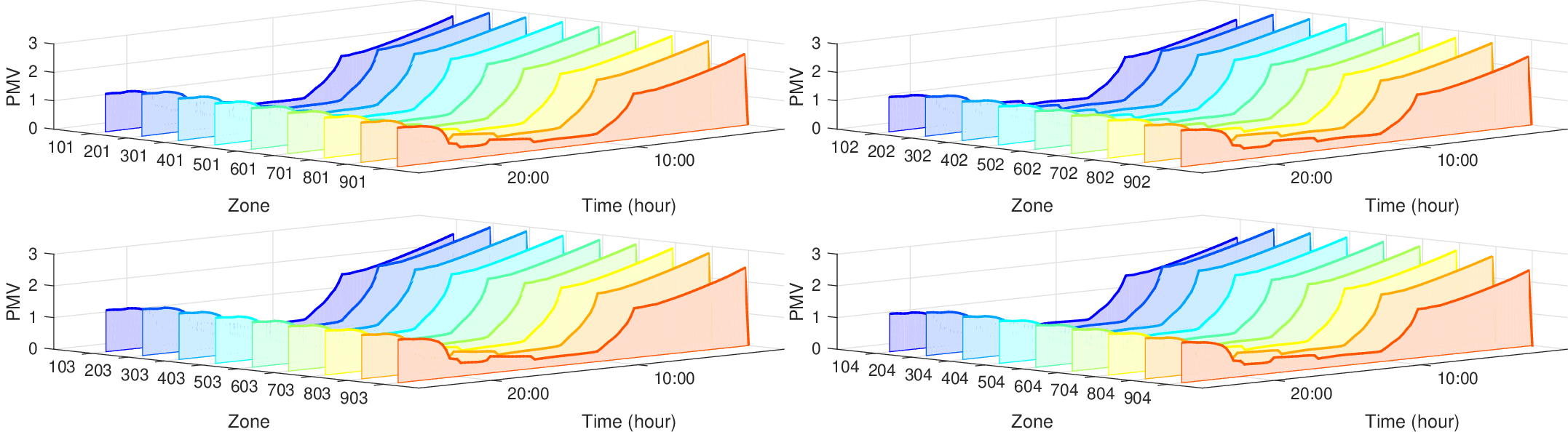} 
	\caption{36 zones PMV.} 
	\label{PMV_distributed} 
    \vskip -0.1in
\end{figure}

\begin{figure}[htbp]
	\centering
	\includegraphics[width=0.48\textwidth]{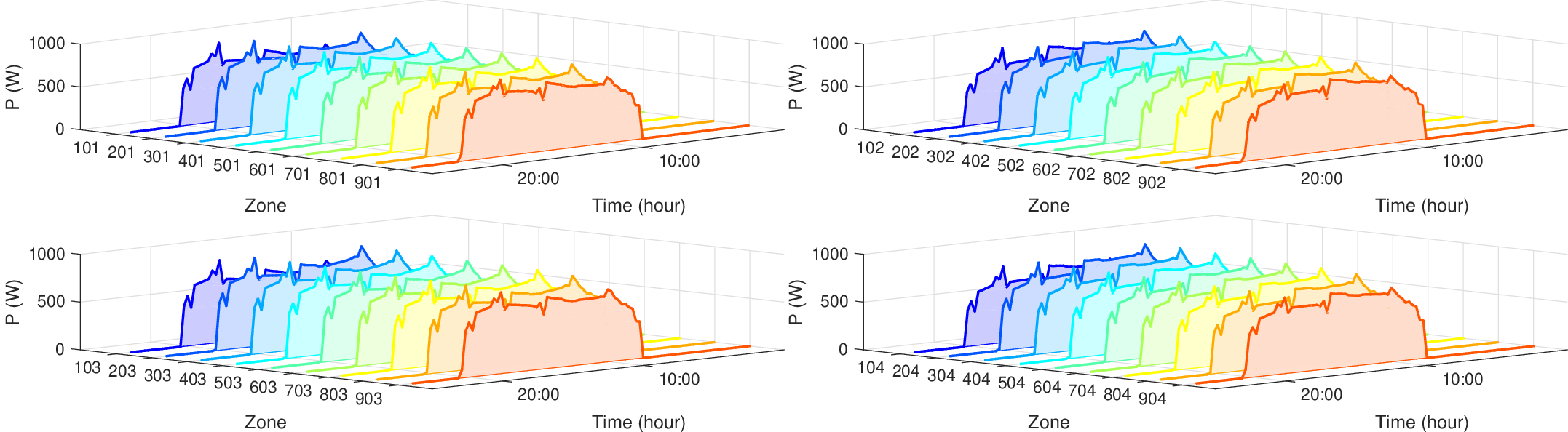} 
	\caption{36 zones energy rate.} 
	\label{Power_distributed} 
    \vskip -0.1in
\end{figure}

\subsection{Comparison methods}
The methods used for comparison are outlined below: 
\subsubsection{Centralized linear}\textcolor{black}{The centralized nonlinear approach loses its comparative value due to its highly time-consuming nature, making it difficult to track or record. As an alternative, the centralized linear method constructs the centralized MPC scheme using a linear model of the PMV index. }
\subsubsection{Centralized PWA}This approach employs the PWA model of the PMV index. Leveraging the characteristics of PWA,  the centralized MPC scheme is built based on a linear model at the relevant operating point during optimization. 
\subsubsection{Distributed nonlinear}

As with (\ref{ADMM}), we introduce an indicator function $\mathrm{I}_S$ and an auxiliary variable $\mathbf{z}$ to derive the ADMM iterations for $\mathcal{P}_1$. The resulting ADMM iterations are as follows:
\begin{subequations}
	\begin{flalign}
	& \mathbf{u}^{\tau+1}_i=\arg \min _{\mathbf{u}_i \in \mathcal{U}_i}\Bigg\{C^{\text{PMV}}_i+C^u_i+\nonumber\\
 &\quad\quad\quad\frac{\rho}{2}\left\|\mathbf{u}_i+\sum\limits_{j\in\mathcal{M}}^{j\neq i} \mathbf{u}^\tau_j-\mathbf{z}^\tau+\frac{\bm{\lambda}^\tau}{\rho}\right\|_2^2 \Bigg\}
 \label{admmy1}\\
	&\mathbf{z}^{\tau+1}=\prod\nolimits_S\bigg(\sum_{i \in \mathcal{M}} \mathbf{u}^{\tau}_i+\frac{\bm{\lambda}^\tau}{\rho}\bigg) \label{10bb}\\
	& \bm{\lambda}^{\tau+1}=\bm{\lambda}^\tau+\rho\bigg(\sum_{i \in \mathcal{M}} \mathbf{u}^{\tau+1}_i-\mathbf{z}^{\tau+1}\bigg).\label{10cc}
	\end{flalign}
 \label{ADMM_nonconvex}
\end{subequations}

 $\mathcal{P}_1$ can be solved directly using (\ref{ADMM_nonconvex}), which requires solving a series of small-scale nonconvex problems. Note that (\ref{ADMM_nonconvex}) employs the original nonlinear formulation of the PMV index without any approximations.

\subsection{Comparison results}
In this part, we compare our method with several approaches to highlight its benefits. 

\subsubsection{Energy cost}
Fig. \ref{e_distributed} illustrates the average power consumption in each zone over a 24-hour period. Table \ref{energy2} summarizes the average power for different methods during the same period. 
\textcolor{black}{
In the centralized method, the PWA approximation is more accurate, resulting in better performance compared with the linear method. In the distributed scheme, the PWA method outperforms the nonlinear scheme due to the solver's limited ability to handle nonlinear problems, while the PWA approximation effectively captures the nonlinear components.
The distributed PWA method ranks second among all methods, only behind its centralized counterpart, because the distributed method tends to yield suboptimal solutions compared to the centralized method. However, as shown in Fig. \ref{e_distributed} and Table \ref{energy2}, the performance gap between the two methods is relatively small, indicating that the distributed PWA method can achieve performance comparable to the centralized method.}

\begin{figure*}[htbp]
\vskip -0.1in
	\centering
	\includegraphics[width=0.8\textwidth]{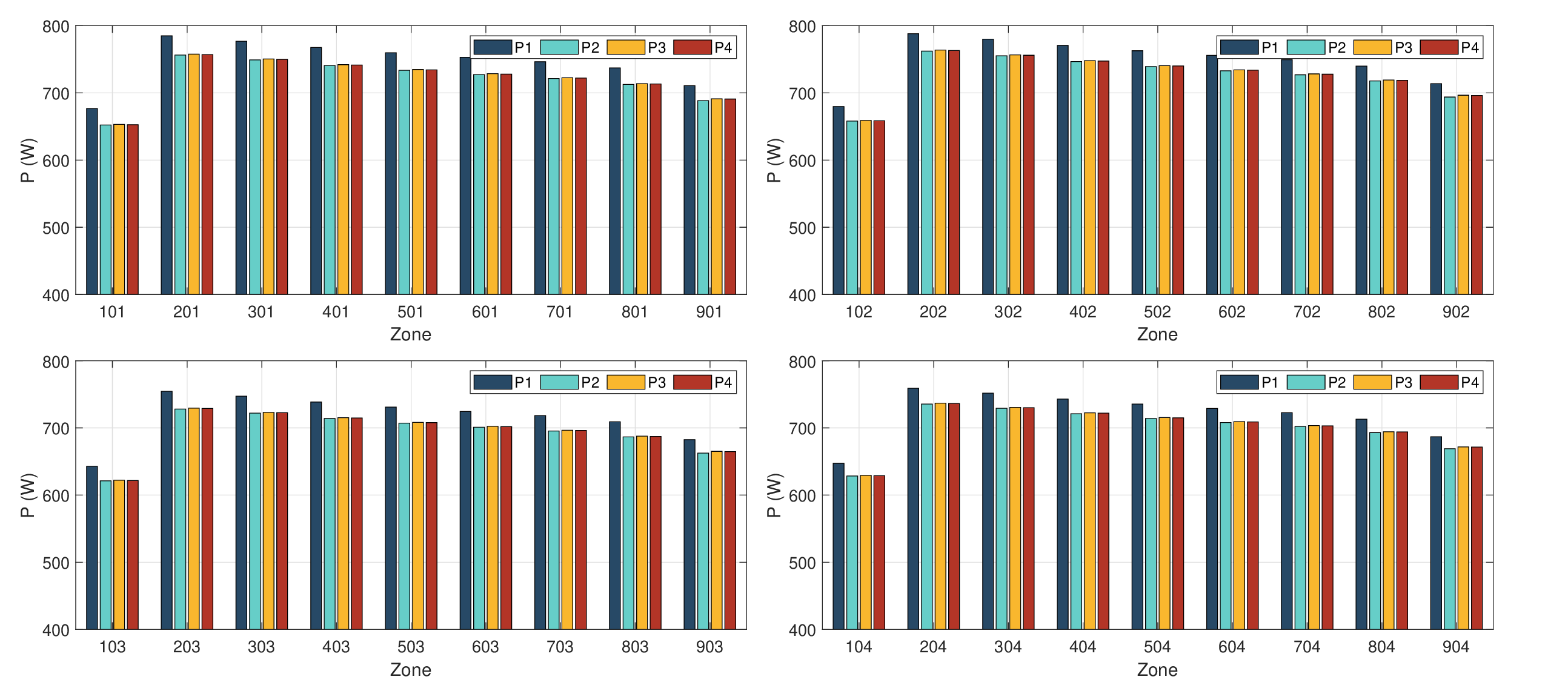} 
	\caption{Energy cost comparison in each zone (P1: Centralized linear, P2: Centralized PWA, P3: Distributed nonlinear, P4: Distributed PWA).} 
	\label{e_distributed} 
    \vskip -0.1in
\end{figure*}

\begin{table}[htbp]
\vskip -0.1in
	\centering
	\caption{Energy Cost Comparison for different methods (unit: $W$)}
	\label{energy2}
	\scriptsize
	\begin{tabular}{cccc}
		\hline %????????
		Centralized linear&Centralized PWA&Distributed nonlinear&Distributed PWA\\
		\hline  %???????
		733.05&709.82&711.23&710.78\\
		\hline%????????				
	\end{tabular}\\
    \vskip -0.1in
\end{table}

\subsubsection{Computational cost}
Table \ref{time3} compares the computational time required by various methods to complete 24-hour closed-loop control. \textcolor{black}{Note that the centralized PWA method approximates the PMV using the PWA technique but employs the PMV affine expression at the current operating point to solve the resulting MPC problem. Thus, the centralized PWA method only needs to solve large-scale convex problems at each operating point. As shown in Table \ref{time3}, both centralized methods solve large-scale convex problems and exhibit similar computational times.} In contrast, both distributed methods significantly improve computational efficiency. The nonlinear distributed method decomposes the resulting large-scale nonlinear problem into a series of smaller nonlinear problems, which reduces the problem size and outperforms the centralized methods. More notably, the PWA distributed method decomposes the resulting large-scale nonconvex problem into a series of small-scale convex problems, leading to a remarkable gain in computational efficiency compared to all other methods. Our proposed method reduces execution time by 86\% compared to the centralized version.

\begin{table}[htbp]
\vskip -0.1in
	\centering
	\caption{Computational Cost Comparison}
	\label{time3}
     \scriptsize
	\begin{tabular}{cccc}
		\hline %????????
		Centralized linear&Centralized PWA&Distributed nonlinear&Distributed PWA\\
		\hline  %???????
		242.89s&244.45s&105.85s&33.57s\\
		\hline%????????				
	\end{tabular}\\
        \footnotesize{$*$ The longest time of sequential computation is employed for the distributed method on one PC.}\\
    \vskip -0.1in
\end{table}

\subsubsection{Comfort cost}
The distribution and variation of PMV over the 24-hour period for several methods are shown in Fig. \ref{boxplot_pmv}. The indoor comfort levels achieved by these methods are similar, and all of them ensure room comfort.
\begin{figure}[htbp]

	\centering
	\includegraphics[width=0.4\textwidth]{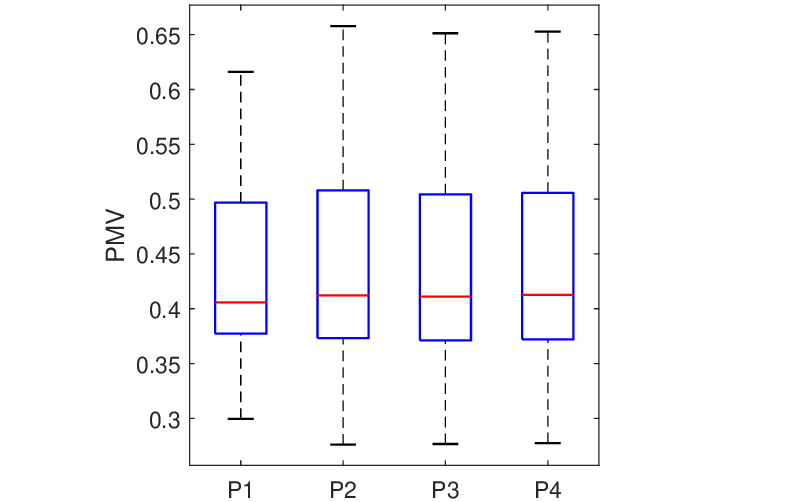} 
	\caption{Comfort Comparison (P1: Centralized linear, P2: Centralized PWA, P3: Distributed nonlinear, P4: Distributed PWA).} 
	\label{boxplot_pmv} 
    \vskip -0.15in
\end{figure}

\section{Discussion}
\label{sec:discussion} 
\textcolor{black}{
This study proposes a PWA-based distributed MPC scheme for optimizing energy consumption and occupant comfort in large buildings. The scheme leverages the decomposition capabilities of the ADMM for effective decomposition and the PWA technique for handling nonlinear components of the PMV index. This section systematically discusses the key strengths and potential limitations of the proposed method.}

\textcolor{black}{
A primary advantage of the proposed method lies in its significant improvement in computational efficiency. Traditional centralized MPC faces exponentially increasing computational complexity when dealing with large-scale building systems (our 36-zone case), especially with coupled constraints and nonlinear PMV models. Our proposed scheme substantially reduces the computational burden by decomposing the large-scale nonconvex problem into a series of smaller convex subproblems, which are then solved in parallel using ADMM. Experimental results demonstrate an 86\% reduction in execution time compared to centralized methods, which is crucial for real-time control in large buildings. Furthermore, the method provides effective handling of nonlinear comfort models. While traditional approaches often rely on linear approximations of the PMV index, sacrificing accuracy, our method preserves the nonlinear characteristics of PMV through PWA approximation, offering more precise comfort modeling. Crucially, the introduced convex ADMM algorithm efficiently handles the resulting PWA nonconvex problem, maintaining high fidelity while avoiding the challenges associated with directly solving nonconvex optimizations. Despite adopting a distributed architecture and PWA approximation, the proposed method also demonstrates high accuracy and comfort assurance. In the 36-zone case study, our method achieved similar energy efficiency to centralized approaches while ensuring high comfort levels (PMV values close to 0), validating its effectiveness in complex multi-zone systems. Lastly, the distributed architecture inherently provides strong scalability. As building sizes or the number of zones increase, our method can effectively cope by simply adding parallel computing resources, avoiding the computational bottlenecks inherent in centralized approaches.}

\textcolor{black}{
However, the proposed method also presents certain limitations worth noting. A major aspect is that its theoretical convergence relies on Assumption 1, which posits that a local optimum lies in the interior of some PWA subregion. While solutions rarely fall exactly on boundaries in many numerical practices, we acknowledge this as a theoretical limitation. If the true optimum lies strictly on a non-smooth subregion boundary, the current theoretical guarantee may not fully ensure finding that exact optimum. Although the algorithm's exploratory mechanism helps mitigate this in practice, a more rigorous treatment of boundary solutions or the pursuit of global optima would necessitate more advanced non-smooth optimization techniques. Another limitation resides in the trade-off between PWA approximation accuracy and complexity. The accuracy of PWA approximation depends on the number and division of its subregions. Increasing the number of subregions can improve approximation fidelity but also increases model complexity and the computational load for solving each subproblem. Adaptively selecting the granularity of the PWA model to achieve an optimal balance between accuracy and computational efficiency remains an ongoing challenge.}

\section{Conclusions}
\label{s6}
This paper proposes a PWA-based distributed scheme for model predictive control in large buildings, focusing on optimizing energy consumption and comfort through PWA-based quadratic programming. To implement this scheme, we design a convex ADMM algorithm that effectively addresses the resulting large-scale nonconvex problems. The convex ADMM algorithm leverages both the ADMM framework and the PWA structure to transform the original large-scale nonconvex optimization problem into a series of smaller convex optimization problems, significantly enhancing computational efficiency. The performance of the proposed distributed schemes is close to that of the centralized approach.

Future work will focus on optimizing energy distribution and comfort in building clusters, as well as considering the impact of various energy hubs on energy supply.

\bibliographystyle{ieeetr}
%\bibliographystyle{cccconf} 
%\nocite{*}
\bibliography{reference}

\end{document}